\documentclass[12pt, a4paper]{amsart}
\usepackage{amsmath,amssymb,amscd,amsfonts}
\newtheorem{theorem}{Theorem}[section]
\newtheorem{rem}[theorem]{Remark}
\newtheorem{defn}[theorem]{Definition}
\newtheorem{lemma}[theorem]{Lemma}

\newtheorem{conjecture}[theorem]{Conjecture}
\newtheorem{corollary}[theorem]{Corollary}

\begin{document}
\thanks{We are in debt to O. Debarre, Z. Jiang, R. Lazarsfeld, J. Koll\'ar and M. Popa for many useful comments. The first author was partially supported by TIMS, NCTS/TPE
and National Science Council of Taiwan.  The second author was
partially supported by NSF Grant 0757897.}
\title[Kodaira dimension of irregular varieties]{Kodaira dimension of irregular varieties}
\author{Jungkai Alfred Chen}
\address{Department of Mathematics \\
National Taiwan University\\ Taipei, 106, Taiwan}
\email{jkChen@math.ntu.edu.tw}
\author{Christopher D. Hacon}
\date{\today}
\address{Department of Mathematics \\
University of Utah\\
155 South 1400 East\\
JWB 233\\
Salt Lake City, UT 84112, USA}
\email{hacon@math.utah.edu}
\begin{abstract} Let $f:X\to Y$ be an algebraic fiber space with general fiber $F$.
If $Y$ is of maximal Albanese dimension, we show that $\kappa (X)\geq \kappa (Y)+\kappa (F)$.
\end{abstract}

\maketitle
%\tableofcontents
\section{Introduction}
Iitaka's $C_{n,m}$ conjecture states:
\begin{conjecture}\label{C-I} Let $f:X\to Y$ be an algebraic fiber space with generic geometric fiber $F$, $\dim X=n$ and $\dim Y=m$. Then $$\kappa (X)\geq \kappa (Y)+\kappa (F).$$
\end{conjecture}
This conjecture is of fundamental importance in the classification of higher dimensional complex projective varieties.
It is known to hold in many important special cases; amongst these:
\begin{enumerate}\item if $\kappa (Y)=\dim Y$, (cf. \cite{Kawamata81} and \cite{Viehweg82}),
\item if $\kappa (F)=\dim F$, (cf. \cite{Kollar87}),
\item if $F$ has a good minimal model (cf. \cite{Kawamata85}), and
\item if $\dim Y=1$, (cf. \cite{Kawamata82}).\end{enumerate}

In this paper we will prove the following:
\begin{theorem}\label{T-I} If $Y$ is of maximal Albanese dimension then Conjecture \ref{C-I} holds.\end{theorem}
Recall that by definition $Y$ is of maximal Albanese dimension if and only if its Albanese morphism $Y\to {\rm Alb}(Y)$ is generically finite.
Note that if $Y$ is a curve of non-negative Kodaira dimension (i.e. genus $g\geq 1$), then $Y$ is of maximal Albanese dimension. Therefore, we obtain an independent proof of the result of \cite{Kawamata82} mentioned above.

We also prove a conjecture of Ueno:
\begin{theorem}\label{main-t}Let $X$ be a smooth complex projective variety of Kodaira dimension $0$.
Then the Albanese morphism $a:X\to A$ is an algebraic fiber space (i.e. it is surjective with connected fibers) and if $F$ is the general fiber of $a$,
then $\kappa (F)=0$.\end{theorem}
\begin{proof} The fact that $a:X\to A$ is an algebraic fiber space  is shown in \cite{Kawamata81}. By \eqref{T-I}, we have $0=\kappa (X)\geq \kappa (F)+\kappa (A)=\kappa (F)$. Since $\kappa (X)\geq 0$, we have $\kappa (F)\geq 0$ and so $\kappa (F)=0$.
\end{proof}

We remark that in \cite{CH09} we showed a weaker version of \eqref{main-t}, namely we showed that $\kappa (F)\leq \dim A$. In this paper we will make use of some of the techniques of \cite{CH09} and in particular
we will make use of multiplier ideal sheaf techniques, Koll\'ar vanishing theorems and the Fourier-Mukai transform.

\section{Preliminaries}
\subsection{Notation}
We work over the field of complex numbers $\mathbb C$.

Let $X$ be a smooth complex projective variety, then the {\bf Kodaira dimension} $\kappa (X)\in \{ -\infty, 0 , 1, \ldots , \dim X \}$ is defined as follows: $\kappa (X)=-\infty$ if $|mK_X|=\emptyset$ for all $m>0$ and otherwise $$\kappa (X)={\rm max}\{\dim \phi _{|mK_X|}(X) | m>0 \}.$$

Let $f:X\to Y$ be a projective morphism of normal varieties, then $f$ is an {\bf algebraic fiber space} if $f$ is surjective with connected general geometric fiber $F=F_{X/Y}$.

If $D$ is a divisor on a projective variety, then ${\bf B}(D)=\cap _{m>0}{\rm Bs}(mD)$ denotes the {\bf stable base locus} of $D$ and  if $V_i\subset |iD|$ is a sequence of linear series such that $V_i\cdot V_j\subset V_{i+j}$, then  ${\bf B}(V_\bullet )=\cap _{m> 0}{\rm Bs}(V_m)$ is the {\bf stable base locus} of $V_\bullet$.
\subsection{Fourier Mukai transform}\label{s-FM}
Let $\mathcal P$ be the normalized Poincar\'e bundle on $A\times \hat A$
and $\hat{\mathcal S}$ be the functor from the category of
$\mathcal O _A$-modules to that of $\mathcal O _{\hat A}$-modules
defined by $$\hat{\mathcal S}(M)=\pi _{\hat A,*}(\mathcal P \otimes \pi _A^*M).$$
By \cite[2.2]{Mukai81} we have an isomorphisms of functors
$$R\mathcal S\circ R\hat{\mathcal S}\cong (-1_A)^*[-g]\qquad {\rm and}\qquad R\hat{\mathcal S}\circ R\mathcal S\cong (-1_{\hat A})^*[-g]$$
so that $R\mathcal S$ gives an equivalence of categories
between $D(\hat A) $ and $D(A)$ (and similarly for $D^b$, $D^-_{qc}$ and $D^-_c$).

By \cite[3.1, 3.8]{Mukai81} we have
$$\Delta _A\circ R\mathcal S\cong ((-1_A)^*\circ R\mathcal S\circ \Delta _{\hat A})[g],\qquad {\rm and}$$
$$R\mathcal S\circ T_{\hat x}^*\cong (\otimes P_{-\hat x})\circ R\mathcal S,\qquad
R\mathcal S\circ (\otimes P_{x})\cong T_{x}^*\circ R\mathcal S.$$

If $L$ is an ample line bundle on $\hat A$, then there is an isogeny $\phi _L:\hat A\to A$ defined by $\phi _L (\hat a)=t^*_{\hat a}L^\vee \otimes L$. Then $\hat {L}=R^0\mathcal S (L)$ is a vector bundle on $A$ such that $\phi _L^*(\hat L ^\vee)\cong L^{\oplus h}$ where $h=h^0(L)$.
We will need the following:
\begin{lemma}\label{l-GD} Let $\mathcal F$ be a quasi-coherent sheaf on $A$. Then
$$D_k(R\Gamma (\mathcal F\otimes \hat L ^\vee))\cong R\Gamma (R\hat {\mathcal S}(D_A(\mathcal F))\otimes L).$$
\end{lemma}
\begin{proof} This is an easy consequence of Groethendieck Duality and the projection formula (see the beginning of the proof of \cite[1.2]{Hacon04}).
\end{proof}
\section{Adjoint ideals}
\subsection{Adjoint ideals}
The following definition is due to Takagi (cf. \cite[3.3]{Eisenstein10}).
\begin{defn}Let $F\subset X$ be a smooth subvariety of codimension $g$ in a smooth variety and $D$ a $\mathbb Q$-divisor whose support does not contain $F$.
Let $\mu :X'\to X$ be the blow up of $X$ along $F$ with reduced exceptional divisor $E$, $\nu :X''\to X'$ a log resolution of $(X',\mu ^{-1}(F)+\mu ^{-1}(D))$ and let $\pi =\mu \circ \nu$. Then we define
$${\rm adj}_F(X,D):=\pi _* \mathcal O _{X''}(\lceil K_{X''/X}-\pi ^{-1}D-g\cdot \pi ^{-1}F+\nu ^{-1}_*E \rceil ).$$\end{defn}
Note that $${\rm adj}_F(X,D)\subset \mathcal J(X,D):=\pi _* \mathcal O _{X''}(\lceil K_{X''/X}-\pi ^{-1}D \rceil ).$$
By \cite[5.1]{Eisenstein10} we have:
\begin{theorem} With the above notation, there is a short exact sequence
$$0\to \mathcal J (X,gF+D)\to {\rm adj}_F(X,D)\to \mathcal J (F,D|_F)\to 0.$$
\end{theorem}

Assume now that $a:X\to A$ is a morphism of smooth projective varieties   $Z\subset A$ is a finite union of closed points with corresponding ideal ${\rm \bf m} _Z \subset \mathcal O _A$ and
that $F=a^{-1}(Z)\subset X$ is a union of fibers.

\begin{corollary}\label{c-basic} Let $H$ and $L$ be Cartier divisors on $A$ and $D_0$ be a Cartier divisor on $X$ such that
\begin{enumerate}
\item $H$ is very ample and $\mathcal O_A(gH)\otimes {\rm \bf m} _Z $ is generated,
\item $D\sim _{\mathbb Q}D_0+\epsilon a^*H$ for some $\epsilon \geq 0$, and
\item $L-(g+\epsilon)H$ is ample (resp. $L-(2g+\epsilon)H$ is ample),
\end{enumerate} then $$H^i(A,a_*(\omega _X(D_0+a^*L)\otimes {\rm adj}_F(X,D)))=0 \qquad \forall i>0$$
(resp. $a_*(\omega _X(D_0+a^*L)\otimes {\rm adj}_F(X,D))$ is generated).
\end{corollary}
\begin{proof}
Let $\mu :X'\to X$ be the blow up of $F$ and let $E$ be the exceptional divisor, then (cf. \cite[9.2.33]{Lazarsfeld04}) $$\mathcal J (X,gF+D)=\mu _* \left( \mathcal J (X',gE+\mu ^* D)\otimes \mathcal O _{X'}(K_{X'/X})\right).$$
Let $G\sim _{\mathbb Q}gH$ be a general $\mathbb Q$-divisor vanishing along $Z\subset A$ to order $g$.
Let $\nu:X''\to X'$ be a log resolution of $(X',E+\mu ^* D+\mu ^* a^*G)$ and $\pi=\mu \circ \nu$, then $\lfloor {\pi}^* a^*G \rfloor =g\nu^* E$ and
$$\mathcal J (X,gF+D)=\pi_*\mathcal O_{X''}(K_{X''/X}-\lfloor {\pi}^*(D+ a^*G) \rfloor ).$$
Since $${\pi}^*(D_0+a^*L)-\lfloor {\pi}^*(D+a^*G)\rfloor \sim _{\mathbb Q} \{{\pi}^*(D+ a^*G) \}+{\pi}^*a^*(L-(g+\epsilon)H),$$ it follows that by Koll\'ar vanishing (cf. \cite[10.15]{Kollar95}) that
$$R^i\pi_*\mathcal O_{X''}(K_{X''/X}+{\pi}^*D_0-\lfloor {\pi}^*(D+ a^*G) \rfloor ) =0\qquad \forall i>0,$$
$$R^ia_*(\omega _X(D_0)\otimes \mathcal J (X,gF+D))=R^i(a\circ  \pi )_*\mathcal O_{X''}(K_{X''}+{\pi}^*D_0-\lfloor {\pi}^*(D+ a^*G) \rfloor ) $$ is torsion free for all $i$
and $$H^i(A,a_* (\omega _X(D_0+a^*L)\otimes \mathcal J (X,gF+D)))=$$
$$H^i(A, (a\circ  \pi)_*\mathcal O_{X''}(K_{X''}+{\pi}^*(D_0+a^*L)-\lfloor {\pi}^*(D+ a^*G) \rfloor ))=0$$ for all $i>0$.
Since $a_*\left( \omega _F(D_0|_F)\otimes \mathcal J(F,D|_F) \right)$ is supported on the finite subset $Z\subset A$, it follows that we have a short exact sequence
  \begin{multline*}0\to a_* \left( \omega _X(D_0+a^*L)\otimes \mathcal J(X,gF+D)\right) \\
\to a_* \left( \omega _X(D_0+a^*L)\otimes {\rm adj}_F(X,D) \right) \\
 \to a_*\left( \omega _F(D_0|_F)\otimes \mathcal J(F,D|_F) \right) \to 0.\end{multline*}
We also have $$H^i(A,a_*\left( \omega _F(D_0|_F)\otimes \mathcal
J(F,D|_F) \right))=0\qquad \forall i>0$$ and hence $H^i(A,a_*\left(
\omega _X(D_0+a^*L)\otimes {\rm adj}_F(X,D) \right) )=0$ for $i>0$
as required.

The remaining statement is a standard consequence of Castelnuovo-Mumford regularity (see eg. \cite[1.8.3]{Lazarsfeld04}).
\end{proof}

From now on we assume that $F\subset X$ is a general fiber of a morphism $a:X\to A$ from a smooth projective variety $X$ to an abelian variety $A$. We may choose the origin $0\in A$
so that $F=X_0$ is the fiber over $0$.

If $V$ is a linear series such that $F$ is not contained
in ${\rm Bs} (V)$ and $D\in V$ is general, we define the ideal
$${\rm adj}_F(X,c\cdot V)={\rm adj}_F(X,c\cdot D).$$

\begin{lemma}\label{gen}
With the above notation, we have $${\rm adj}_F(X,c\cdot V) \supset {\rm
adj}_F(X,c\cdot D'),$$ for any $D' \in V$ such that $F\not \subset {\rm Supp}(D')$.
\end{lemma}
\begin{proof} Standard.\end{proof}
If $V_i\subset |iL|$ is a sequence of linear series such that $V_i\cdot V_j\subset V_{i+j}$ and $F\not \subset {\bf B}(V_\bullet )$, then we let
 $${\rm adj}_F(X,c\cdot ||V_\bullet ||)=\bigcup _{i>0}{\rm adj}_F(X,\frac ci\cdot D_i)$$
where $D_i\in V_i$ is general. Recall that as $X$ is Noetherian, we have ${\rm
adj}_F(X,c\cdot ||V_\bullet ||)= {\rm adj}_F(X,\frac ci\cdot D_i)$ for
all $i >0$ sufficiently divisible.

Fix $m>0$ and $D_F\in |mK_F|$. By a result of Viehweg (cf. \cite{Viehweg83}, \cite{Viehweg95} and \cite{Kollar95}), for any rational number $\epsilon
>0$, the map
$$|t(mK_X+\epsilon a^*H)|\to |tmK_F|$$ is surjective for all $t>0$ sufficiently divisible. Let $$V_{t,\epsilon ,D_F}=\{G\in
|t(mK_X+\epsilon a^*H)|\ {\rm s.t.}\ G|_F=tD_F \}.$$ Note that
$V_{i,\epsilon ,D_F}\cdot V_{j,\epsilon ,D_F}\subset V_{i+j,\epsilon
,D_F}$ and so we define the ideal $${\rm adj}_F(X,||V_{\bullet
,\epsilon ,D_F}||):=\bigcup _{t>0}{\rm adj}_F(X,\frac 1t\cdot V_{t,\epsilon
,D_F}).$$

\subsection{Descending chains of adjoint ideals}

\begin{lemma}
For any $0<\epsilon '<\epsilon$ we have
$${\rm adj}_F(X,||V_{\bullet ,\epsilon ' ,D_F}||)\subset {\rm adj}_F(X,||V_{\bullet ,\epsilon ,D_F}||).$$
\end{lemma}

\begin{proof}
We can pick $t>0$ sufficiently divisible such that $${\rm
adj}_F(X,||V_{\bullet ,\epsilon ,D_F}||)={\rm adj}_F(X,\frac 1t\cdot
V_{t,\epsilon ,D_F})\qquad {\rm and} $$ $${\rm adj}_F(X,||V_{\bullet ,\epsilon '
,D_F}||)={\rm adj}_F(X,\frac 1t\cdot V_{t,\epsilon ' ,D_F}).$$
Moreover, we have $$V_{t, \epsilon ', D_F} +|t(\epsilon-\epsilon
') a^*H| \subset V_{t,\epsilon,D_F}.$$

Pick a general $G \in V_{t, \epsilon ', D_F}$ so that ${\rm
adj}_F(X,\frac 1t\cdot V_{t,\epsilon ' ,D_F})={\rm adj}_F(X,\frac
1t\cdot G)$. If $G' \in
|t(\epsilon-\epsilon ')a^* H|$ is general, then ${\rm adj}_F(X,\frac 1t\cdot
G)={\rm adj}_F(X,\frac 1t\cdot (G+G'))$. By Lemma \ref{gen}, we have
$${\rm adj}_F(X,\frac 1t\cdot (G+G')) \subset {\rm adj}_F(X,\frac 1t\cdot
V_{t,\epsilon ,D_F}).$$
\end{proof}
\begin{lemma}\label{l-e} There exists a constant $0< \epsilon_0 \ll 1$ such that $$a_* (\omega _X ^{\otimes m+1}\otimes {\rm adj}_F(X,||V_{\bullet ,\epsilon
,D_F}||))=a_* (\omega _X ^{\otimes m+1}\otimes {\rm
adj}_F(X,||V_{\bullet ,\epsilon_0  ,D_F}||))$$
for all $\epsilon< \epsilon_0$.  Moreover, if we denote this coherent sheaf by $\mathcal F _{||D_F||}$, then \begin{enumerate}
\item $H^i(A,\mathcal F
_{||D_F||}\otimes \mathcal O _A(L))=0$ if $i>0$ and $L-gH$ is ample;
\item $\mathcal F _{||D_F||}\otimes \mathcal O _A(L)$ is generated if
$L-2gH$ is ample.\end{enumerate}\end{lemma}
\begin{proof} By \eqref{c-basic}, one sees that if $L-2gH$ is ample, then there exists a rational number $0<\epsilon _0\ll 1 $ such that $a_* (\omega _X ^{\otimes m+1}\otimes {\rm adj}_F(X,||V_{\bullet ,\epsilon
,D_F}||))\otimes \mathcal O _A(L)$ is generated for all $0<\epsilon \leq \epsilon _0$.
Therefore, if the inclusion $$a_* (\omega _X ^{\otimes m+1}\otimes {\rm adj}_F(X,||V_{\bullet ,\epsilon '
,D_F}||))\subset a_* (\omega _X ^{\otimes m+1}\otimes {\rm adj}_F(X,||V_{\bullet ,\epsilon
,D_F}||))$$ is strict for some $0<\epsilon ' <\epsilon \leq \epsilon _0$, then
\begin{multline*}h^0(a_* (\omega _X ^{\otimes m+1}\otimes {\rm adj}_F(X,||V_{\bullet ,\epsilon '
,D_F}||))\otimes \mathcal O _A(L))\\
<h^0( a_* (\omega _X ^{\otimes m+1}\otimes {\rm adj}_F(X,||V_{\bullet ,\epsilon
,D_F}||))\otimes \mathcal O _A(L)).\end{multline*}
Since $H^0(A,a_* (\omega _X ^{\otimes m+1}\otimes {\rm adj}_F(X,||V_{\bullet ,\epsilon
_0 ,D_F}||))\otimes \mathcal O _A(L))$ is finite dimensional, we may assume that the above inclusions are equalities for
any $0<\epsilon '< \epsilon \leq \epsilon _0\ll 1$.

The remaining statements then follow
from \eqref{c-basic}.
\end{proof}
\begin{lemma}\label{l-inv} The sheaf $\mathcal F_{||D_F||}$ is independent of $H$.\end{lemma}
\begin{proof} Standard.\end{proof}

{\bf Convention.} Fixed $H$ and $D_F$ as above, we may
pick $0< \epsilon \ll 1$, $t=t(\epsilon) \gg 0$ and $D \in |t(mK_X+a^*\epsilon H)|$
such that $\mathcal F_{||D_F||}$ {\bf is computed by} $(\epsilon, t,
D)$, that is $$\mathcal F_{||D_F||}=a_* (\omega _X ^{\otimes
m+1}\otimes {\rm adj}_F(X,\frac{1}{t}\cdot D )).$$

\subsection{Inductive limits}
Let $\pi _k :A_k\to A$ be multiplication by $2^k$ so that $\pi
_k=\pi_{k-1}\circ \pi _1$ and $A_k\cong A$. We let $X_k=X \times
_AA_k$ and $\pi_{X,k}:X_k\to X$, $a_k:X_k\to A_k$. We have a
sequence of inclusions of $\mathcal O _A$ modules
$$  \mathcal O _A\subset {\pi _1}_* \mathcal O _{A_1}\subset {\pi _2}_* \mathcal O _{A_2}\cdots \subset  \mathcal O _{\infty}:=\bigcup _{i\geq 0}{\pi _i}_* \mathcal O _{A_i}.$$
Then $\mathcal O _{\infty}$ is the direct sum of all topological
trivial line bundles $P\in {\rm Pic}^0(A)$ such that $\pi _k
^*P=\mathcal O _{A_k}$ for some $k>0$. We may think of $A_\infty$
as the inverse limit of the system $\ldots A_k\to A_{k-1}\to
\ldots$.

We have the following diagram
$$\begin{CD}
X_k @>{\pi_{X,k}}>> X \\
@VV{a_k}V @VV{a}V \\
A_k @>{\pi_k}>> A \\
\end{CD}$$

Fix once and for all a very ample divisor $H$ on $A$, and let $H_k$ be the
corresponding divisor on $A_k$ so that $\pi _k^*H\equiv 2^{2k}H_k$ (see for example  \cite[Section 6, Corollary 3]{Mumford70}). Let $\mathcal F^k_{||D_F||}\in {\rm
Coh}(A_k)$ be the sheaves defined analogously to $\mathcal F_{||D_F||}$ (cf. \eqref{l-e}) so that
$$\mathcal F^k_{||D_F||}={a_k}_*(\omega _{X_k}^{\otimes m+1}\otimes {\rm adj}_{F_k}(X_k,\frac 1 t \cdot D_k))$$ where $F_k=a_k^{-1}(0)\cong F$, and $D_k$ is a general member of $V^k_{t,\epsilon,D_F}:=|t(mK_{X_k}+a_k^*\epsilon H_k)|$ such that $D_k|_{F_k}=tD_F$, $0<\epsilon \ll 1$ and $t\gg 0$.

Let $\mathcal
V^k_{||D_F||}$ be the coherent sheaves on $A_k$ defined using
$\mathcal J (X_k,||V_{\bullet , \epsilon , D_F}||)$ instead of ${\rm
adj}_{F_k}(X_k,||V_{\bullet , \epsilon , D_F}||)$ so that $$\mathcal
V^k_{||D_F||}={a_k}_*(\omega _{X_k}^{\otimes m+1}\otimes \mathcal
J(X_k,||V^k_{\bullet , \epsilon , D_F}||)).$$

We now define coherent sheaves on $A$
$$\mathcal G^k_{||D_F||}:={\pi _k}_*\mathcal F^k_{||D_F||}, \quad \mathcal W^k_{||D_F||}:={\pi _k}_*\mathcal V^k_{||D_F||}.
$$
Since ${\rm adj}_{F_k}(X_k,||V^k_{\bullet , \epsilon , D_F}||)\subset \mathcal J (X_k,||V^k_{\bullet , \epsilon , D_F}||)$, we have $\mathcal G^k_{||D_F||} \subset \mathcal
W^k_{||D_F||}$ for all $k>0$.
\begin{lemma}\label{l-in} For all
$k>0$ we have inclusions
$$ \mathcal G^k_{||D_F||}\subset \mathcal G^{k+1}_{||D_F||}\qquad  {\rm and}\qquad  \mathcal W^k_{||D_F||}\subset \mathcal W^{k+1}_{||D_F||}\subset \mathcal V _{||D_F||}\otimes {\pi _{k+1}}_* \mathcal O _{A_{k+1}}.$$
\end{lemma}

\begin{proof}
By induction on $k$, to prove the first inclusion, it suffices to prove that $\mathcal
F_{||D_F||} \subset {\pi_1}_* \mathcal F^1_{||D_F||}$.

We may assume that there exist $\epsilon_0, t_0$ and
$D \in |t_0(mK_X+a^*\epsilon_0 H)|$ such that $\mathcal
F_{||D_F||}$ is computed by $(\epsilon_0, t_0, D)$. Let
$D_1=\pi_{X,1}^*D$ on $X_1$. It is clear that
$$D_1 \in
\pi_{X,1}^*|t_0(mK_X+a^*\epsilon_0 H)|
\subset
|\pi_{X,1}^*(t_0(mK_X+a^*\epsilon_0 H))|=|
t_0(mK_{X_1}+a_1^*\epsilon_0 \pi _1^*H)|.$$
We may furthermore assume
that with this choice of $\epsilon_0$ and $ t_0$, a general element $D'_1 \in
|t_0(mK_{X_1}+a_1^*\epsilon_0 \pi _1^*H)|$ computes $\mathcal
F^1_{||D_F||}$.
We have \begin{multline*}{\rm adj}_{F_1}(X_1,\frac 1 t\cdot  D_1')\supset {\rm adj}_{F_1}(X_1,\frac 1 t\cdot D_1)\\
\supset {\rm adj}_{\pi _{X,1}^{-1}(F)}(X_1,\frac 1 t \cdot D_1) = \pi _{X,1}^*{\rm adj}_{F}(X,\frac 1 t \cdot D).\end{multline*}
 The desired inclusion now follows by pushing forward (via $a$) the inclusion
\begin{multline*}\omega _X^{\otimes m+1}\otimes {\rm adj}_{F}(X,\frac 1 t\cdot D)\subset
{\pi _{X,1}}_*(\omega _{X_1}^{\otimes m+1}\otimes \pi _{X,1}^*{\rm adj}_{F}(X,\frac 1 t\cdot D))\\ \subset {\pi _{X,1}}_*(\omega _{X_1}^{\otimes m+1}\otimes{\rm adj}_{F_1}(X_1,\frac 1 t \cdot D_1')).\end{multline*}

The proof of the inclusions $\mathcal{W}^k _{||D_F||}\subset\mathcal W^{k+1}_{||D_F||}\subset \mathcal V _{||D_F||}\otimes {\pi _{k+1}}_* \mathcal O _{A_{k+1}}$ is
similar.
\end{proof}
\section{Proofs of Theorems \ref{T-I} and \ref{main-t}}

\begin{lemma}\label{l-V}  If $\kappa (X)=0$ and $P_{m+1}(X)=1$, then $\mathcal V^k_{||D_F||}$ is a unipotent vector
bundle and for all $k>0$ we have
$$\mathcal G^k_{||D_F||}\subset \mathcal W^k_{||D_F||}=\mathcal V_{||D_F||}\otimes {\pi _k}_*\mathcal O _{A_k}.$$
\end{lemma}
\begin{proof}
The inclusion $\mathcal G^k_{||D_F||}\subset \mathcal W^k_{||D_F||}$ has already been observed above.
%The first part follows easily from flat base change the inclusions
%${\rm adj}_{\frac 1 i \cdot ||V_{i , \epsilon , D_F}||}\subset \mathcal J (\frac 1 i \cdot ||V_{i , \epsilon , D_F}||)$.
Since $X_k\to X$ is \'etale, $\kappa (X_k)=0$ and hence $1\geq P_{m+1}(X_k)\geq P_{m+1}(X)=1$. By the proof of \cite[5.4]{Hacon04},  for all $k>0$ we have that $\mathcal V^k_{||D_F||}$ is a
unipotent vector bundle on $A_k$ of rank $r$ (where $0<r\leq P_{m+1}(F)$).
Hence $\mathcal V^k_{||D_F||}$ admits a filtration with successive quotients isomorphic to  $\mathcal
O_{A_k}$.  Since ${\pi_k}_* \mathcal O_{A_k}=\bigoplus _{P\in {\rm Ker}(A_k\to \hat A_k)}P$ is a homogeneous
vector bundle of rank $2^{2kg}$, it follows that ${\pi_k}_*
\mathcal V^k_{||D_F||}$ is homogeneous of rank $2^{2kg} r$.

On the other hand, $\mathcal V_{||D_F||} \otimes {\pi _k}_*\mathcal O _{A_k}$
has rank $2^{2kg} r $ and so the inclusion $\mathcal{W}^k _{||D_F||}\subset \mathcal V _{||D_F||}\otimes {\pi _k}_* \mathcal O _{A_k}$ of homogeneous vector
bundles of the same rank gives the required equality.
\end{proof}
\begin{corollary}\label{c-V} Let $p_k:A_k\to A_{k-1}$ be the induced morphism. Then the natural map $p_k^*\mathcal V ^{k-1}_{||D_F||}\to  \mathcal V ^{k}_{||D_F||}$ is an isomorphism.
\end{corollary}
\begin{proof} Consider the isomomorphism $${p_k}_*p_k^*\mathcal V ^{k-1}_{||D_F||}\cong \mathcal V ^{k-1}_{||D_F||}\otimes  {p_k}_*\mathcal O_{A_k}\cong {p_k}_*\mathcal V ^{k}_{||D_F||}$$ where the first equality is given by the projection formula and last equality follows from \eqref{l-V}.
Since $p_k$ is \'etale, we then have a homomorphism $p_k^*\mathcal V ^{k-1}_{||D_F||}\to  \mathcal V ^{k}_{||D_F||}$
which
is a isomomorphism of homogeneous vector bundles.
\end{proof}
\begin{rem} In order to prove \eqref{main-t}, we would like to argue as follows:
Let $$\mathcal G^\infty _{||D_F||}=\bigcup _{k>0}\mathcal G^k _{||D_F||}\subset \mathcal V_{||D_F||}\otimes \mathcal O _{\infty}.$$
We will show (cf. \eqref{l-van}) that for any sufficiently ample line bundle $L$ on $\hat A$, we have $$H^i(A,\mathcal G^\infty _{||D_F||}\otimes \hat L ^\vee )=0\qquad \forall i>0.$$ By \eqref{l-GD}, we
have that $$R\Gamma (R\hat{\mathcal S}(D_A(\mathcal G^\infty _{||D_F||}))\otimes L)\cong D_k(R\Gamma (\mathcal G^\infty _{||D_F||})\otimes \hat L ^\vee )$$
which (following the proof of \cite[1.2]{Hacon04}) should imply that $R\hat{\mathcal S}(D_A(\mathcal G^\infty _{||D_F||}))$ is a sheaf  so that
$$V^0(\mathcal G^\infty _{||D_F||})\supset V^1(\mathcal G^\infty _{||D_F||})\supset \ldots \supset V^g(\mathcal G^\infty _{||D_F||}).$$
But then as $\mathcal G^\infty _{||D_F||}\ne 0$, by section \ref{s-FM}, we have
$H^0(A,\mathcal G^\infty _{||D_F||}\otimes P )\ne 0$ for some $P\in {\rm Pic }^0(A)$.
In turn this implies that $$H^0(X_k,\omega _{X_k}^{\otimes m+1}\otimes{\rm adj}_{F_k}(X_k,\frac 1 i \cdot V^k_{i,\epsilon , D_F})\otimes P)= H^0(A,\mathcal G^k_{||D_F||}\otimes P)\ne 0.$$
Let $0\ne \sigma \in H^0(X_k,\omega _{X_k}^{\otimes m+1}\otimes{\rm adj}_{F_k}(X_k,\frac 1 i \cdot V^k_{i,\epsilon , D_F})\otimes P)$, then
from the inclusion $${\rm adj}_{F_k}(X_k,\frac 1 i \cdot V^k_{i,\epsilon , D_F})\hookrightarrow \mathcal J (D_F) =\mathcal O _F(-D_F),$$ it follows that $\sigma |_F\in \omega _F^{\otimes m+1}(-D_F)$. If $\kappa (F)>0$, then $D_F$ varies in a positive dimensional family so that  $\kappa (\omega _{X_k}^{\otimes m+1}\otimes P)>0$.
This in turn implies that $\kappa (X)>0$.
Therefore, if $\kappa (X)=0$, we conclude that $\kappa (F)=0$.
\end{rem}
\begin{lemma}\label{l-ses} Let $\mathcal Q$ be the cokernel of the inclusion $\mathcal G ^\infty _{||D_F||}\hookrightarrow \mathcal V_{||D_F||} \otimes \mathcal O _{\infty}$. Then $\mathcal Q$ is a coherent sheaf.
\end{lemma}
\begin{proof} Recall that there is an inclusion $
\mathcal F^k _{||D_F||} \hookrightarrow  \mathcal V^k _{||D_F||}$.
Let $\mathcal R _k$ be the cokernel of this inclusion. Since
$\pi_k$ is \'etale, $R^1{\pi_k}_*\mathcal F^k _{||D_F||}=0$.
Thus, pushing forward by $\pi_k$,
we obtain short exact sequences
$$0\to \mathcal G ^k _{||D_F||}\to \mathcal V_{||D_F||} \otimes {\mu _k}_* \mathcal O _{A_k}\to \mathcal Q _k\to 0.$$

Consider now the morphism $p_k:A_k\to A_{k-1}$. We have the
following commutative diagram of short exact sequences:
$$\begin{CD}
0@>>> {p_k}^*\mathcal F^{k-1} _{||D_F||}@>>> {p_k}^*\mathcal V^{k-1}
_{||D_F||}@>>> {p_k}^*\mathcal R
_{k-1}@>>>  0\\
@. @V{\alpha'}VV @V{\beta'}VV @V{\gamma'}VV @. \\
0@>>>  \mathcal F^k _{||D_F||}@>>>  \mathcal V^k _{||D_F||}@>>>  \mathcal R _{k}@>>> 0.\\
\end{CD}$$
Notice that by the proof of \eqref{l-in} $\alpha'$ is inclusion and by  \eqref{c-V} $\beta '$ is an isomorphism. Therefore $\gamma '$ is surjective. Pushing forward via $p_k$, one sees that the composition $\mathcal R _{k-1}\to {p_k}^*\mathcal R
_{k-1} \to \mathcal R _{k}$ is surjective.
Therefore, we have a commutative diagram
$$\begin{CD}
0@>>> \mathcal F^{k-1} _{||D_F||}@>>> \mathcal V^{k-1}
_{||D_F||}@>>> \mathcal R
_{k-1}@>>>  0\\
@. @V{\alpha}VV @V{\beta}VV @V{\gamma}VV @. \\
0@>>>  {p_k}_*\mathcal F^k _{||D_F||}@>>>  {p_k}_*\mathcal V^k _{||D_F||}@>>>  {p_k}_*\mathcal R _{k}@>>> 0\\
\end{CD}$$
where %$\alpha$ is an inclusion %by the proof of Lemma \ref{l-V}. By \eqref{l-V}, the second map
%$\beta$ is an isomorphism %. The diagram is commutative by our construction. Therefore,
%and
$\gamma$ is
surjective. Pushing forward via $\pi_{k-1}$, we then have the
surjection $\phi _{k-1}:\mathcal Q _{k-1}\to \mathcal Q _{k}$ for all $k>0$.
Since $\mathcal Q _k$ has finite length (for all $k$), we have that
$\phi _k$ is an isomorphism for $k \gg 0$. Let  $\mathcal
Q=\mathcal
Q_k$ for $k\gg 0$ be the resulting coherent sheaf. Taking the direct limit of
the push-forward of the above diagram, we thus
complete the proof.
%Consider the short exact sequences $$0\to \mathcal G ^k _{||D_F||}\to \mathcal V \otimes {\mu _k}_* \mathcal O _{A_k}\to \mathcal Q _k\to 0.$$ Since the inclusion ${\rm adj}_F(X,||V_{\bullet , \epsilon, D_F ||})\to \mathcal J(X,||V_{\bullet , \epsilon, D_F ||})$ is an isomorphism away from $F$, one sees that $\mathcal Q _k$ is supported on $0\in A$. Consider the short exact sequences on $A_k$ $$\begin{CD} 0@>>>  p_k^*\mathcal F _{||D_F||}@>>>  p_k^*\mathcal V _{||D_F||}@>>>  p_k^*\mathcal R _{k-1}@>>> 0\\@. @VVV @VVV @VVV @. \\ 0@>>> \mathcal F _{||D_F||}@>>> \mathcal V _{||D_F||}@>>> \mathcal R _k@>>>  0\end{CD}$$ where $\mathcal R _k$ is supported on $0\in A_k$, $p_k^*\mathcal R _{k-1}$ is supported on $p_k^{-1}(0)\subset  A_k$ and $p_k:A_k\to A_{k-1}$ is the induced morphism.By \eqref{l-V}, the second inclusion is an isomorphism. Since $p_k^*\mathcal F _{||D_F||}\to \mathcal F _{||D_F||}$ is an inclusion, it follows that the homomorphism $p_k^*\mathcal R _{k-1}\to \mathcal R _{k}$ is surjective.Notice also that on a neighborhood of $0\in A_k$, we have that $p_k^*\mathcal R _{k-1}\cong\mathcal R _{k} $. It follows that the induced homomorphism $$\mathcal R _{k-1}\to {p_k}_*{p_k}^* \mathcal R _{k-1}\to {p_k}_*\mathcal R _{k}$$ is surjective. Since $R^1{\pi _k}_*\mathcal F _{||D_F||}=0$, we have that  ${\pi _k}_*\mathcal R _k=\mathcal Q _k$. But then there is a surjection $\mathcal Q _{k-1}\to \mathcal Q _{k}$. Since  $\mathcal Q _k$ has finite length, we have that $\mathcal Q _k\cong \mathcal Q$ for all $k\gg 0$.
\end{proof}
\begin{lemma}\label{l-van}If $L$ is any sufficiently ample line bundle on $\hat A$, then $H^i(A,\mathcal G ^\infty _{||D_F||}\otimes \hat L ^\vee)=0$ for $i>0$.
\end{lemma}
\begin{proof}
By \cite[III.2.9]{Hartshorne77}, since $\mathcal G^\infty _{||D_F||}=\varinjlim \mathcal G^k_{||D_F||}$, then
$$H^i(A,\mathcal G^\infty _{||D_F||}\otimes \hat L ^\vee )=\varinjlim H^i(A,\mathcal G^k_{||D_F||}\otimes \hat L ^\vee ).$$ It suffices therefore to check that for any $L$, we have $H^i(A,\mathcal G^k_{||D_F||}\otimes \hat L ^\vee )=0$ for $i>0$ and $k\gg 0$.
Let $\phi _L:\hat A \to A$, then $\phi _L^*\hat L ^\vee \cong L^{\oplus h}$ where $h=h^0(L)$, and so $\hat L ^\vee$ is an ample vector bundle. If $A'_k=A_k\times _A \hat A$, then we have the following
diagram:
$$\begin{CD}
A'_k @>{\psi}>> \hat A \\
@V{\rho}VV  @VV{\phi_L}V\\
A_k @>{\pi_k}>> A
\end{CD}$$

 Notice that:
\begin{enumerate}
\item
$\deg (\rho)=h$.
\item $\rho^*\pi _k^* \hat L ^\vee=\psi^* \phi_L^* \hat L
^\vee= \psi^* L^{\oplus h}$ is a direct sum of line bundles.
\item
Since $\rho^* H_k$ is an ample line bundle, $3 \rho^* H_k$ is very ample. \item Since $3H_k \otimes {\rm \bf m}_0$ is generated, so is
$3g\rho^* H_k\otimes {\rm \bf m}_{\rho ^{-1}(0)}$.
\item Since $\hat L ^\vee-\delta H$ is ample for some $\delta >0$ and $\pi _k ^* H\equiv 2^{2k} H_k$, it follows that $\pi _k^* \hat L ^\vee
\otimes \mathcal O_{A_k}(-3gH_k)$ is ample for $k\gg 0$. Thus $\rho^* (\pi _k^* \hat L ^\vee
\otimes \mathcal O_{A_k}(-3gH_k))$ is a direct sum of ample line
bundles.\end{enumerate}
Let $X'_k=X_k\times _{A_k}A'_k$, $a'_k:X'_k\to A'_k$ and $F'_k\subset X'_k$ be the inverse image of $F_k$ (so that $F'_k={a'_k}^{-1}(\rho ^{-1}(0))$ has $h$ distinct irreducible components isomorphic to $F_k$).
It is easy to see that
$$\rho^* (\mathcal F^k_{||D_F||})={a'_k}_*(\omega _{X'_k}^{\otimes m+1}\otimes {\rm adj}_{F'_k}(X'_k,\frac 1 t \cdot D'_k)),$$
where $D'_k\in |tm(K_{X'_k}+{a'_k}^*\rho ^*\epsilon H_k)|$ is the pull-back of a general element   $D_k\in |tm(K_{X_k}+a_k^*\epsilon H_k)|$ such that $D_k|_{F_k}=tD_F$, $0<\epsilon \ll 1$ and $t\gg 0$.
%Since  $3\rho ^* H_k$ is very ample and $\rho^* (\pi _k^* \hat L ^\vee \otimes \mathcal O_{A_k}(-3ghH_k))$ is a direct sum of ample line bundles, following
Since $\frac 1 t \cdot D_k'\sim _{\mathbb Q}m(K_{X'_k}+(\rho \circ a'_k)^*\epsilon H_k)$ and $\pi _k^* \hat L ^\vee
\otimes \mathcal O_{A_k}(-(3g+m\epsilon )H_k)$ is ample for $\epsilon\ll 0$, by \eqref{c-basic}, one sees that
$$H^i(A'_k, \rho^* (\mathcal F^k_{||D_F||}\otimes \pi _k^*  \hat L ^\vee) )=H^i(A'_k, {a'_k}_*(\omega _{X'_k}^{\otimes m+1}\otimes {\rm adj}_{F'_k}(X'_k,\frac 1 t \cdot D'_k))\otimes \rho ^*\pi _k^*  \hat L ^\vee)=0$$ for all $i>0.$
Since $\rho$ is \'etale, by the projection formula we have that
$$H^i(A_k, \mathcal F^k_{||D_F||}\otimes  \pi _k^*  \hat L ^\vee )=0\qquad \forall i>0.$$
Since $\pi _k$ is \'etale, and ${ \pi _k }_* \mathcal
F_{||D_F||}=\mathcal G ^k_{||D_F||}$, it now follows easily that
$H^i(A,\mathcal G^k_{||D_F||}\otimes \hat L ^\vee )=0$ for $i>0$ and
$k\gg 0$.
\end{proof}
\begin{lemma} We have $R^{i}\hat{\mathcal S}(D_A (\mathcal G ^\infty_{||D_F||}))=0$ for $i\ne g$.
\end{lemma}
\begin{proof}Since $\mathcal V_{||D_F||} \otimes \mathcal O _{\infty }=\bigoplus _{P\in {\rm Ker}(\hat \pi _k)}\mathcal V_{||D_F||} \otimes P$, then (cf. Subsection \ref{s-FM}) $$R\hat{\mathcal S}(D_A (\mathcal V_{||D_F||} \otimes \mathcal O _{\infty }))=R^{g}\hat{\mathcal S}(D_A (\mathcal V_{||D_F||} \otimes \mathcal O _{\infty }))=$$
$$\bigoplus _{P\in {\rm Ker}(\hat \pi _k)}R^{g}\hat{\mathcal S}(D_A (\mathcal V_{||D_F||} \otimes P))=\bigoplus _{P\in {\rm Ker}(\hat \pi _k)}\widehat{\mathcal V_{||D_F||} ^\vee \otimes P^\vee }=\bigoplus _{x\in {\rm Ker} (\pi _k)}T_x^*( \widehat {\mathcal V_{||D_F||} ^\vee } ) .$$
Therefore, by \eqref{l-ses},  we get an exact sequence
\begin{multline*}0\to R^{g-1}\hat{\mathcal S}(D_A (\mathcal G ^\infty_{||D_F||}))\to \widehat { D_A( \mathcal Q)} \to R^{g}\hat{\mathcal S}({ D_A (\mathcal V_{||D_F||} \otimes \mathcal O _{\infty })})\\
\to R^{g}\hat{\mathcal S}(D_A (\mathcal G ^\infty_{||D_F||}))\to 0.\end{multline*}
In particular, $R^{i}\hat{\mathcal S}(D_A (\mathcal G ^\infty_{||D_F||}))=0$ for $i\not\in\{ g-1,g \}$ and
$\widehat { D_A( \mathcal Q )} \to R^{g}\hat{\mathcal S}({ D_A (\mathcal V_{||D_F||} \otimes \mathcal O _{\infty })})$ is a homomorphism from a unipotent vector bundle to a direct sum of Artinian modules.
Let $$\tau = {\rm Im}\left( \widehat { D_A( \mathcal Q )} \to R^{g}\hat{\mathcal S}({ D_A (\mathcal V_{||D_F||} \otimes \mathcal O _{\infty })})\right)$$ and consider the induced short exact sequence
$$0\to \tau \to R^{g}\hat{\mathcal S}({ D_A (\mathcal V_{||D_F||} \otimes \mathcal O _{\infty })})\to R^{g}\hat{\mathcal S}(D_A (\mathcal G ^\infty_{||D_F||}))\to 0.$$
Since $\tau \subset R^{g}\hat{\mathcal S}({ D_A (\mathcal V_{||D_F||} \otimes
\mathcal O _{\infty })}) $ are direct sums of Artinian modules, so
is $R^{g}\hat{\mathcal S}(D_A (\mathcal G ^\infty_{||D_F||}))$. This
implies in particular $H^i(\hat{A},R^{g}\hat{\mathcal S}(D_A
(\mathcal G ^\infty_{||D_F||}))\otimes L)=0$ for $i>0$ and any line
bundle $L$.

Since $R^{g-1}\hat{\mathcal S}(D_A (\mathcal G
^\infty_{||D_F||}))\subset \widehat{D_A(\mathcal Q)}$ is a coherent sheaf, we then
have that $R^{g-1}\hat{\mathcal S}(D_A (\mathcal G
^\infty_{||D_F||}))\otimes L$ is generated and
$H^i(\hat{A},R^{g-1}\hat{\mathcal S}(D_A (\mathcal G
^\infty_{||D_F||}))\otimes L)=0$ for $i>0$, and $L$ sufficiently
ample.

Following the proof of \cite[1.2]{Hacon04}, we consider the spectral sequence
$$E_2^{i,j}=R^i\Gamma (R^j\hat {\mathcal S}(D_A (\mathcal G ^\infty _{||D_F||}))\otimes L)$$ abutting to $R^{i+j}\Gamma (R\hat {\mathcal S}(D_A(\mathcal G^\infty _{||D_F||}))\otimes L)$.
Since $E_2^{i,j}=0$ if $(i,j)\not \in \{(0,g-1),(0,g) \}$, it
follows that the above spectral sequence degenerates at the $E_2$
term.

By \eqref{l-GD}, we have that $$D_k(R\Gamma (\mathcal G ^\infty _{||D_F||}\otimes \hat L^\vee))\cong R\Gamma (R \hat {\mathcal S}(D_A(\mathcal G ^\infty _{||D_F||}))\otimes L).$$
By \eqref{l-van}, we have that $R\Gamma
(\mathcal G ^\infty _{||D_F||}\otimes \hat L^\vee)=R^0\Gamma
(\mathcal G ^\infty _{||D_F||}\otimes \hat L^\vee)$.
Therefore, $R^{g-1}\Gamma (R\hat {\mathcal S}(D_A(\mathcal G^\infty _{||D_F||}))\otimes L)=0$ and hence $E_2^{0,g-1}=0$.

But then $R^0\Gamma (R^{g-1}\hat {\mathcal S}(D_A (\mathcal G ^\infty _{||D_F||}))\otimes L)=0$ so that
 $R^{g-1}\hat {\mathcal S}(D_A (\mathcal G ^\infty _{||D_F||}))\otimes L=0$.
\end{proof}
\begin{corollary} $\mathcal G ^\infty_{||D_F||}\cong
\mathcal V_{||D_F||} \otimes \mathcal O _{\infty}$.\end{corollary}
\begin{proof} Since $\widehat{D_A(\mathcal Q  )}\to R^{g}\hat{\mathcal S}(D_A (\mathcal V_{||D_F||} \otimes \mathcal O _{\infty}))$ is an injection from a vector bundle to a direct sum of torsion sheaves,  $\widehat{D_A(\mathcal Q  )}=0$.
Thus $R^{g}\hat{\mathcal S}(D_A (\mathcal G ^\infty_{||D_F||}))\cong
R^{g}\hat{\mathcal S}(D_A (\mathcal V_{||D_F||} \otimes \mathcal O _{\infty}))$ and the claim easily follows cf. Section \eqref{s-FM}.
\end{proof}
We will now prove Theorem \ref{main-t}:
\begin{theorem}\label{t-k} Let $a:X\to A$ be an algebraic fiber space, $F$ the general fiber and $A$ an abelian variety. If $\kappa (X)=0$, then $\kappa (F)=0$.\end{theorem}
\begin{proof} %By a standard reduction, we may assume that $P_1(X)=0$.
Since $\kappa (X)=0$, $S_t:=\{ P\in {\rm Pic} ^0(X) |h^0(\omega _X^{\otimes t}\otimes P)\ne 0\}=\{\mathcal O _X \}$ for all $t>0$ sufficiently divisible cf. \cite[3.2]{CH02}. Since $\mathcal V_{||D_F||} \subset a_*(\omega _X^{\otimes m+1})$ is a unipotent vector bundle, then for any $m+1$ sufficiently divisible, we have that
$H^0(A,\mathcal G ^\infty_{||D_F||})\cong
H^0(A,\mathcal V_{||D_F||} \otimes \mathcal O _{\infty})\ne 0$.
Since $$H^0(A,\mathcal G ^\infty_{||D_F||})=\varinjlim H^0(A, \mathcal G ^k _{||D_F||}),$$ we have that
$$H^0(A, \mathcal G ^k _{||D_F||}  )\ne 0\qquad \forall k\gg 0.$$
Pick any $k$ such that
$$H^0(X_k,\omega _{X_k}^{\otimes m+1}\otimes{\rm adj}_F(X,||V_{\bullet , \epsilon, D_F}||))=
H^0(A_k,\mathcal F ^k_{||D_F||})=H^0(A, \mathcal G ^k _{||D_F||}  )\ne 0.$$
Since $\kappa (X)=0$, we have $\kappa (X_k)=0$ and hence $H^0(X_k,\omega _{X_k}^{\otimes m+1}\otimes{\rm adj}_F(X,||V_{\bullet , \epsilon, D_F}||))\cong \mathbb C$.
Let $0\ne \sigma \in H^0(X_k,\omega _{X_k}^{\otimes m+1}\otimes{\rm adj}_F(X,||V_{\bullet , \epsilon, D_F}||))$, then as the homomorphism ${\rm adj}_F(X,||V_{\bullet , \epsilon, D_F}||)\to \mathcal J (F,D_F)$ is surjective, $\sigma|_F$ vanishes along $D_F$.
If $\kappa (F)>0$, then we may assume that $D_F$ varies in a positive dimensional family $|D_F|$ and hence that $h^0(X_k,\omega _{X_k}^{\otimes m+1})>0$.
This is a contradiction and hence $\kappa (F)=0$.
\end{proof}

We will now prove Theorem \ref{T-I}:
\begin{theorem}Let $f: X \to Y $ be an algebraic fiber space with $Y$ of maximal
Albanese dimension and general fiber $F$. Then $\kappa(X) \ge
\kappa(Y)+ \kappa(F)$.
\end{theorem}
\begin{proof}
Assume that $\kappa (X)<0$, and let $a_X:X\to A_X$ (resp. $a_Y: Y
\to A_Y$)  be the albanese morphism of $X$ (resp. of $Y$). Let
$\alpha_X: X \to X'$ (resp. $\alpha_Y: Y \to Y'$) be the Stein
factorization of $a_X$ (resp.  $a_Y$). Since $f_* \mathcal O_X=
\mathcal O_Y$, it follows that $X\to Y'$ is the Stein
factorization of $X \to A_Y$. But then $f': X' \to Y'$
is an algebraic fiber space and the following diagram
is commutative

$$\begin{CD}
X @>{\alpha_X}>> X' @>{\beta_X}>> A_X \\
@V{f}VV @V{f'}VV @VVV\\
Y @>{\alpha_Y}>> Y' @>{\beta_Y}>> A_Y.
\end{CD}
$$

Since $Y$ is of maximal Albanese dimension, $\alpha_Y$ is
birational. Let $G$ be the general fiber of $\alpha_X: X \to X'$ and
$E$ be the general fiber of $X' \to Y'$. Then $F$ maps
surjectively on to $E$ with general fiber $G$. By \cite[3.1]{Lai09} (applied to $X\to X'$),
one sees that $\kappa (G)<0$. By the easy addition formula (applied to $F\to E$), one sees
that $\kappa(F) <0$. Thus the Theorem is true if $\kappa (X)<0$.

We may therefore assume that $\kappa (X)\geq 0$. We will follow the
arguments of \cite[6.4]{Mori85}.

We will first show that the statement holds when $\kappa (Y)=0$. In
this case, by Kawamata's Theorem (cf. \cite{Kawamata81}), we may assume that $Y$ is an
abelian variety. Let $G$ be the general fiber of $\phi:X \to Z$, the
Iitaka fibration of $X$, then $\kappa (G )=0$. Let $K=f(G)$, then $K$ is an abelian subvariety of $Y$ and $f|_G: G \to
K$ is the Albanese map of $G$. The general fiber of $f|_G$ is
$F':=F \cap G$. By \eqref{t-k}, we have $\kappa(F')=0$. On the
other hand, $F'$ can is the general fiber of $ \phi_F: F \to
 Z$. By easy addition, we have $$ \kappa(F) \le \kappa(F')+\dim
({\rm im} (\phi_F)) \le \dim Z = \kappa(X).$$

% the general fiber of $f_|G: G \to K$ is also $F$ and so $\kappa
%(F)=0$ by \eqref{main-t}.

We now consider the general case. Let $Y \to A_Y$ be the Albanese
map of $Y$ and  $Y'$ be the Stein factorization of $Y \to A_Y$. Then
$Y \to Y'$ is birational. We may thus replace $Y$ by $Y'$ and
assume that $Y \to A_Y$ is finite. By
\cite[Theorem 13]{Kawamata81}, after replacing $X$ and $Y$ by \'etale covers, we may
assume that there is an abelian variety $K$ and a variety of general type
$W$ such
that $Y$ is isomorphic to $W\times K$.
Note that $\kappa (W)=\dim W =\kappa (Y)$. Let $E$ be the general fiber of the algebraic
fiber space $X \to W$. By \cite[Theorem 3]{Kawamata81}, one has $$\kappa(X) \ge
\kappa(E)+ \kappa(W) = \kappa(E) + \kappa(Y).$$ Notice also that
$f|_E: E \to K$ has general fiber $F$. Since $\kappa (K)=0$, by what we have shown above, we have  $$\kappa(E) \ge \kappa(F)+\kappa (K)=\kappa(F).$$ Hence we have
$$ \kappa(X) \ge \kappa(Y)+\kappa(F),$$ as required.

\end{proof}

\end{document}